\documentclass[leqno,english]{amsart}

\usepackage{hyperref}   
\usepackage[square,numbers]{natbib}     
\usepackage{amsthm}     
\usepackage{amsmath}    
\usepackage{amssymb}    
\usepackage{amsxtra}    
\usepackage{eucal}      
\usepackage[bbgreekl]{mathbbol}   
\usepackage[all]{xy}    

\theoremstyle{plain}

\newtheorem{thm}{Theorem}

\swapnumbers                        
\newtheorem{prop}[subsection]{Proposition}
\newtheorem{lemm}[subsection]{Lemma}

\theoremstyle{definition}

\newtheorem{constr}[subsection]{Construction}
\newtheorem{rem}[subsection]{Reminders}

\numberwithin{subsubsection}{section}

\numberwithin{equation}{subsubsection}

\title[The cotriple resolution of differential graded algebras]{The cotriple resolution\\
of differential graded algebras}

\date{February 4, 2016}

\author{Benoit Fresse}
\address{Univ. Lille, CNRS, UMR 8524 - Laboratoire Paul Painlev\'e, F-59000 Lille, France}
\email{Benoit.Fresse@math.univ-lille1.fr}

\thanks{This research is supported in part by grant ANR-11-BS01-002 ``HOGT'' and by Labex ANR-11-LABX-0007-01 ``CEMPI''.
I am grateful to Paul Goerss, who introduced me to the homotopical constructions studied in this paper during my first visit
at Northwestern University a long time ago. I thank Victor Turchin and Thomas Willwacher who provided me the motivation
to write down this article. I thank the referee for instructive comments on analogues of the result
of this article in the field of stable homotopy theory}


\DeclareMathOperator{\kk}{\mathbb{k}}   
\DeclareMathOperator{\NN}{\mathbb{N}}   
\DeclareMathOperator{\ZZ}{\mathbb{Z}}   

\DeclareMathOperator{\ACat}{\mathcal{A}}            
\DeclareMathOperator{\FCat}{\mathcal{F}}            
\DeclareMathOperator{\MCat}{\mathcal{M}}            
\DeclareMathOperator{\Simp}{\mathit{s}\mathcal{S}\mathit{et}}   

\DeclareMathOperator{\dg}{\mathit{dg}}      
\DeclareMathOperator{\cosimp}{\mathit{c}}   
\DeclareMathOperator{\simp}{\mathit{s}}     

\DeclareMathOperator{\Map}{\mathtt{Map}}    
\DeclareMathOperator{\Mor}{\mathtt{Mor}}    

\DeclareMathOperator{\Id}{\mathit{Id}}      
\DeclareMathOperator{\id}{\mathit{id}}      

\DeclareMathOperator{\II}{\mathbb{I}}   
\DeclareMathOperator{\SSym}{\mathbb{S}}     
\DeclareMathOperator{\ESym}{\mathbb{E}}     

\DeclareMathOperator{\DGC}{\mathtt{C}}          
\DeclareMathOperator{\DGE}{\mathtt{E}}          
\DeclareMathOperator{\DGH}{\mathtt{H}}          
\DeclareMathOperator{\DGL}{\mathtt{L}}          
\DeclareMathOperator{\DGN}{\mathtt{N}}          

\DeclareMathOperator{\Tot}{\mathtt{Tot}}        
\DeclareMathOperator{\Res}{\mathtt{Res}}        
\DeclareMathOperator{\sk}{\mathtt{sk}}          

\DeclareMathAlphabet{\mathsfit}{OT1}{cmss}{m}{sl}   
\DeclareMathOperator{\EOp}{\mathsfit{E}}        
\DeclareMathOperator{\IOp}{\mathsfit{I}}        
\DeclareMathOperator{\MOp}{\mathsfit{M}}        
\DeclareMathOperator{\NOp}{\mathsfit{N}}        
\DeclareMathOperator{\POp}{\mathsfit{P}}        
\DeclareMathOperator{\ROp}{\mathsfit{R}}        
\DeclareMathOperator{\ComOp}{\mathsfit{Com}}    

\begin{document}

\begin{abstract}
We consider the cotriple resolution of algebras over operads in differential graded modules.
We focus, to be more precise, on the example of algebras over the differential graded Barratt-Eccles operad,
and on the example of commutative algebras.
We prove that the geometric realization of the cotriple resolution (in the sense of model categories)
gives a cofibrant resolution functor on these categories
of differential graded algebras.
\end{abstract}

\maketitle

\section*{Introduction}

The purpose of this paper is to explain the definition of cofibrant resolutions of differential graded algebras
from the classical cotriple resolution.
We consider, to be more precise, the particular case of algebras over an $E_{\infty}$-operad and the case of commutative algebras.
We deal with algebras in differential graded modules in both cases.
We use the generic notation $\POp$ for an operad that governs one of these categories of algebras.
We also use the notation of this operad $\POp$ to specify the category of dg-algebras in which we form our constructions.
We set $\POp = \EOp$ in the $E_{\infty}$-case, because we adopt this notation $\EOp$ for a certain model
of an $E_{\infty}$-operad (the Barratt-Eccles operad actually),
and we set $\POp = \ComOp$ in the commutative algebra case.
We also use the prefix dg to abbreviate the expression ``differential graded''.
We are mostly motivated by the applications of our constructions to dg-algebras over the operads $\POp = \ComOp,\EOp$,
because these categories of dg-algebras define models
for the homotopy of spaces (rationally or completed at a prime depending on the context).
We mostly refer to~\cite{Mandell,Sullivan} for these applications of dg-algebras in homotopy theory.

Let $A$ be an object in any of our categories of dg-algebras.
The cotriple resolution of $A$ is a simplicial object $R_{\bullet} = \Res^{\POp}_{\bullet}(A)$,
canonically associated to $A$, and which forms a free algebra dimension-wise.
The cotriple resolution is also endowed with a natural augmentation $\epsilon: \Res^{\POp}_{\bullet}(A)\rightarrow A$,
where we identify the dg-algebra $A$ with a constant simplicial object.

We use a general definition of the geometric realization, which makes sense in the context of model categories,
to retrieve an object of our category of dg-algebras $|R_{\bullet}|$
from this simplicial dg-algebra $R_{\bullet} = \Res^{\POp}_{\bullet}(A)$.
We just use that our categories of dg-algebras inherit a model structure
in order to perform this construction.
The general theory of model categories implies that $R = |\Res^{\POp}_{\bullet}(A)|$ is cofibrant as an object of our category of dg-algebras,
but does not give any information on the homotopy type of this object.
The actual goal of this paper is to prove that, in the context of our categories of dg-algebras, the augmentation
of the cotriple resolution $\epsilon: \Res^{\POp}_{\bullet}(A)\rightarrow A$
induces a weak-equivalence when we pass to the geometric realization:
\begin{equation*}
\epsilon: |\Res^{\POp}_{\bullet}(A)|\xrightarrow{\sim} A.
\end{equation*}
We conclude from this result that the dg-algebra $R = |\Res^{\POp}_{\bullet}(A)|$ defines a cofibrant resolution of our initial object~$A$.

Let us mention that we have an analogous statement in the context of algebras over operads in spectra (see~\cite[\S VII.3.3]{ElmendorfAl},
\cite[Proposition 6.11]{HarperHess}). The proofs of this topological analogue of our result in the cited references
crucially rely on the observation that the category of algebras over operads in spectra
is cotensored over simplicial sets with function objects
inherited from the category of spectra.
We do not have such a result for the categories of algebras over operads in dg-modules.
We therefore have to use a new approach in order to establish our statement in the dg-module setting.

\medskip
The geometric realization and the cotriple resolution are universal and are endowed with good categorical properties
which make them particularly useful.
To cite one significant application, we consider function spaces $\Map_{\POp}(-,-)$, $\POp = \ComOp,\EOp$,
which we associate to our model categories
of dg-algebras.
We have the interchange formula
$\Map_{\POp}(|\Res^{\POp}_{\bullet}(A)|,B) = \Tot\Map_{\POp}(\Res^{\POp}_{\bullet}(A),B)$,
where $\Tot(-)$ denotes the Bousfield-Kan totalization functor on the category of simplicial sets~\cite[\S X.3]{BousfieldKan}.
The Bousfield-Kan spectral sequence gives a general method to compute the homotopy of this total space
from the cohomotopy of the cosimplicial set $\pi_*\Map_{\POp}(\Res^{\POp}_{\bullet}(A),B)$,
which in turn can be identified with a cosimplicial deformation complex
associated to our objects.
The equivalence $|\Res^{\POp}_{\bullet}(A)|\xrightarrow{\sim}A$ implies that this spectral sequence
computes the homotopy of the function space $\Map_{\POp}(A,B)$
with our initial object $A$
as source.
This statement has to be confronted with the results of~\cite{BousfieldResolutions}, which provide an interpretation, in terms of completions,
of the outcome of homotopy spectral sequences associated to the (co)triple resolution
in general model categories. Our statement basically implies that this completion functor reduces to the identity in the contex of our categories
of dg-algebras.

In the case of commutative algebras, we may use this homotopy spectral sequence
to give an alternate proof of the main result
of~\cite{BousfieldPetersonSmith}, which gives an interpretation of function spaces $\Map_{\Simp}(X,Y)$
of simplicial sets $X,Y\in\Simp$
in terms of division functors on commutative algebras. In the case of algebras over an $E_{\infty}$-operad,
our spectral sequence represents an algebraic counterpart
of the classical unstable Adams spectral sequence,
and we have a generalization of the result of~\cite{BousfieldPetersonSmith}
in this context (see~\cite{ChataurKuribayashi})
which may again be proved by relying on our constructions.

\medskip
This paper is organized as follows. We address the case of algebras over an $E_{\infty}$-operad first.
We have to fix a model for the class of $E_{\infty}$-operads.
We take the differential graded Barratt-Eccles operad (we refer to~\cite{BergerFresse} for detailed study of this operad
in the differential graded context).
We actually deal with non-unitary algebras first. We therefore consider a non-unitary version of the Barratt-Eccles operad,
which we denote by $\EOp$, and which has a zero module as arity zero term $\EOp(0) = 0$.
We devote the first section of this paper to the definition of a cosimplicial framing
on the category of algebras over this operad.
We explain the definition of the cofibrant resolutions from the cotriple resolution, and we give the proof of our main result
in the context of algebras over the Barratt-Eccles operad in the second section.
We address the extension of our results in the context of algebras over a unitary version of the Barratt-Eccles operad
in the third section. We study the cotriple resolution of commutative algebras
in the fourth section, and we outline a generalization of our results for algebras
over arbitrary operads in a concluding section.

\section{Simplicial framings and cosimplicial framings}\label{Background}

We work in the base category of dg-modules over a fixed ground ring $\kk$, where a dg-module is a module $K$
equipped with a decomposition $K = \oplus_{n\in\ZZ}K_n$
and a differential $\delta: K\rightarrow K$
that lowers degrees by one.
We equip this category, which we denote by $\dg\kk$, with its standard model structure,
where the weak-equivalences are the morphisms which induce an isomorphism
in homology, and the fibrations are the degree-wise surjective maps (see for instance~\cite[\S 2.3]{Hovey}
or~\cite[\S 11.1.9]{OperadModules} and~\cite[\S II.5]{FresseBook}).
We also provide $\dg\kk$ with its usual symmetric monoidal category structure, where the symmetry isomorphism
involves a sign that reflects the standard commutation rules
of differential graded algebra (see for instance~\cite[\S II.5]{FresseBook}).
We assume for simplicity that the ground ring $\kk$ is a field
all through this paper, so that all degree-wise
injective morphisms of dg-modules
form a cofibration in this base model category (see the general characterization of cofibrant objects
in categories of dg-modules in~\cite[Proposition 2.3.9]{Hovey},
or adapt the arguments of~\cite[Proposition II.5.1.10]{FresseBook}).
We refer to the book~\cite{Hovey} and to~\cite[\S II.1-4]{FresseBook} for an overall survey of the concepts
of the theory of model categories which we use in this paper.
We also refer to the books~\cite{OperadModules,FresseBook} for an account of the applications
of these concepts in the context of operads.

We use the notation ${}_{\EOp}\dg\kk$ for the category of $\EOp$-algebras in the category of dg-modules $\dg\kk$.
We also adopt the notation $\ESym: \dg\kk\rightarrow{}_{\EOp}\dg\kk$
for the free $\EOp$-algebra functor on $\dg\kk$
which represents the left adjoint of the obvious forgetful functor $\omega: {}_{\EOp}\dg\kk\rightarrow\dg\kk$.
We then consider the full subcategory $\FCat({}_{\EOp}\dg\kk)\subset{}_{\EOp}\dg\kk$
whose objects are the free $\EOp$-algebras $F = \ESym(K)$, $K\in\dg\kk$.

We equip ${}_{\EOp}\dg\kk$ with its standard model structure which we determine from the base category of dg-modules
by assuming that the forgetful functor $\omega: {}_{\EOp}\dg\kk\rightarrow\dg\kk$
creates weak-equivalences and fibrations.
We aim to define a simplicial framing functor on the category ${}_{\EOp}\dg\kk$
and a cosimplicial framing functor $-\otimes\Delta^{\bullet}$
on the subcategory $\FCat({}_{\EOp}\dg\kk)\subset{}_{\EOp}\dg\kk$.
We consider, for this purpose, the normalized chain complex $\DGN_*(\Delta^n)$ of the simplices $\Delta^n\in\Simp$, $n\in\NN$.
We use that these dg-modules $\DGN_*(\Delta^n)$ form coalgebras over the Barratt-Eccles operad $\EOp$ (see~\cite{BergerFresse}).
We dually get that the conormalized cochain complexes of the simplices $\DGN^*(\Delta^n)$
form $\EOp$-algebras.

We then use the observation that the Barratt-Eccles operad is equipped with a Hopf operad structure
to obtain that the tensor product $B^{\Delta^n} = B\otimes\DGN^*(\Delta^n)$,
where $B$ is any algebra over this operad $B\in{}_{\EOp}\dg\kk$,
inherits a natural $\EOp$-algebra structure,
for each $n\in\NN$.
We moreover have the following statement:

\begin{prop}\label{Background:SimplicialFraming}
Let $B\in{}_{\EOp}\dg\kk$. The tensor products $B^{\Delta^n} = B\otimes\DGN^*(\Delta^n)$, $n\in\NN$,
define a simplicial frame of the object $B$ in the category of $\EOp$-algebras ${}_{\EOp}\dg\kk$.
\end{prop}

\begin{proof}
We are left to proving that the tensor products $B^{\Delta^n} = B\otimes\DGN^*(\Delta^n)$, $n\in\NN$, fulfill the properties of a simplicial frame
in the category of dg-modules (since the forgetful functor $\omega: {}_{\EOp}\dg\kk\rightarrow\dg\kk$
creates limits, fibrations, and weak-equivalences), but this statement is well known in this setting (see for instance~\cite[\S II.6]{FresseBook}).
\end{proof}

We now construct a left adjoint of the functor $B^{\Delta^n} = B\otimes\DGN^*(\Delta^n)$, $n\in\NN$,
to define our cosimplicial framing functor on the category of free $\EOp$-algebras.

\begin{constr}[The cosimplicial framing]\label{Background:CosimplicialFramingConstruction}
For an object $F = \ESym(K)$ of our category $\FCat({}_{\EOp}\dg\kk)$, and for any $n\in\NN$, we basically set:
\begin{equation*}
\ESym(K)\otimes\Delta^n := \ESym(K\otimes\DGN_*(\Delta^n)),
\end{equation*}
where we take the free $\EOp$-algebra on the object $K\otimes\DGN_*(\Delta^n)\in\dg\kk$.

Then we consider a morphism $\phi = \phi_f: \ESym(K)\rightarrow\ESym(L)$ which we determine from a morphism
of the ambient category $f: K\rightarrow\ESym(L)$,
for any $K,L\in\dg\kk$, by using the definition of free algebras.
Recall that the free algebra $\ESym(M)$
associated to any object $M\in\dg\kk$
has an expansion of the form $\ESym(M) = \bigoplus_{r=1}^{\infty}(\EOp(r)\otimes M^{\otimes r})_{\Sigma_r}$,
where, for each $r\in\NN$, we consider the coinvariants of the tensor product $\EOp(r)\otimes M^{\otimes r}\in\dg\kk$
under the diagonal action of the symmetric group on $r$ letters $\Sigma_r$,
and we use the notation $\oplus$ for the coproduct in the category $\dg\kk$.

For each $r\in\NN$, we consider the morphism $f: K\rightarrow(\EOp(r)\otimes L^{\otimes r})_{\Sigma_r}$,
which defines the component of weight $r$ of $f: K\rightarrow\ESym(L)$,
and we form the composite
\begin{multline*}
K\otimes\DGN_*(\Delta^n)\xrightarrow{f\otimes\id}(\EOp(r)\otimes L^{\otimes r})_{\Sigma_r}\otimes\DGN_*(\Delta^n)
\\
\begin{aligned}
\simeq & (\EOp(r)\otimes L^{\otimes r}\otimes\DGN_*(\Delta^n))_{\Sigma_r}
\xrightarrow{\Delta_*}(\EOp(r)\otimes\EOp(r)\otimes L^{\otimes r}\otimes\DGN_*(\Delta^n))_{\Sigma_r}
\\
\simeq & (\EOp(r)\otimes L^{\otimes r}\otimes\EOp(r)\otimes\DGN_*(\Delta^n))_{\Sigma_r}
\xrightarrow{\rho_*}(\EOp(r)\otimes L^{\otimes r}\otimes\DGN_*(\Delta^n)^{\otimes r})_{\Sigma_r}
\\
\simeq & (\EOp(r)\otimes(L\otimes\DGN_*(\Delta^n))^{\otimes r})_{\Sigma_r},
\end{aligned}
\end{multline*}
where $\Delta_*$ is the morphism induced by the coproduct of the Barratt-Eccles operad $\Delta: \EOp(r)\rightarrow\EOp(r)\otimes\EOp(r)$,
and $\rho_*$ is the morphism yielded by the coaction of the Barratt-Eccles operad
on the normalized complex of the simplex $\rho: \EOp(r)\otimes\DGN_*(\Delta^n)\rightarrow\DGN_*(\Delta^n)^{\otimes r}$.
We consider the morphism
\begin{equation*}
f^{\sharp}: K\otimes\DGN_*(\Delta^n)\rightarrow\ESym(L\otimes\DGN_*(\Delta^n))
\end{equation*}
defined by these composites on the components of the free $\EOp$-algebra $\ESym(L\otimes\DGN_*(\Delta^n))$.
We now set $\phi_f\otimes\Delta^n := \phi_{f^{\sharp}}$
\begin{equation*}
\phi_f\otimes\Delta^n := \phi_{f^{\sharp}}: \ESym(K\otimes\DGN_*(\Delta^n))\rightarrow\ESym(L\otimes\DGN_*(\Delta^n)),
\end{equation*}
where we consider the morphism of free $\EOp$-algebras $\phi_{f^{\sharp}}$
which we associate to this morphism $f^{\sharp}$
in the base category category~$\dg\kk$.
\end{constr}

We have the following claim:

\begin{prop}\label{Background:CosimplicialFramingAdjunction}
The functor of the previous paragraph $-\otimes\Delta^n: \FCat({}_{\EOp}\dg\kk)\rightarrow\FCat({}_{\EOp}\dg\kk)$
is left adjoint, on the subcategory $\FCat({}_{\EOp}\dg\kk)\subset{}_{\EOp}\dg\kk$,
to the functor $(-)^{\Delta^n}: {}_{\EOp}\dg\kk\rightarrow{}_{\EOp}\dg\kk$
of Proposition~\ref{Background:SimplicialFraming},
for any $n\in\NN$.
\end{prop}

\begin{proof}
We have the adjunction relations
\begin{multline*}
\Mor_{{}_{\EOp}\dg\kk}(\ESym(K),B\otimes\DGN^*(\Delta^n))
\xrightarrow[(1)]{\simeq}\Mor_{\dg\kk}(K,B\otimes\DGN^*(\Delta^n))
\\
\xrightarrow[(2)]{\simeq}\Mor_{\dg\kk}(K\otimes\DGN_*(\Delta^n),B)
\xrightarrow[(3)]{\simeq}\Mor_{{}_{\EOp}\dg\kk}(\ESym(K\otimes\DGN_*(\Delta^n)),B)
\end{multline*}
at the object level, for any $K\in\dg\kk$. The morphism $\eta: \ESym(K)\rightarrow\ESym(K\otimes\DGN_*(\Delta^n))\otimes\DGN^*(\Delta^n)$
that defines the unit of this adjunction
relation and corresponds to the identity of the object $B = \ESym(K\otimes\DGN_*(\Delta^n))$
under our bijections (1-3),
is yielded by the map
$K\rightarrow K\otimes\DGN_*(\Delta^n)\otimes\DGN^*(\Delta^n)\subset\ESym(K\otimes\DGN_*(\Delta^n))\otimes\DGN^*(\Delta^n)$,
where we consider the morphism induced by the trace map $T: \kk\rightarrow\DGN_*(\Delta^n)\otimes\DGN^*(\Delta^n)$
on our tensor product,
followed by the morphism induced by the canonical embedding of the dg-module~$K\otimes\DGN_*(\Delta^n)$
in the free $\EOp$-algebra~$\ESym(K\otimes\DGN_*(\Delta^n))$.

Let $\phi = \phi_f: \ESym(K)\rightarrow\ESym(L)$ be a morphism of free $\EOp$-algebras
associated to a morphism $f: K\rightarrow\ESym(L)$
in the base category of dg-modules.
We easily check that the morphism $\phi\eta: \ESym(K)\rightarrow\ESym(L\otimes\DGN_*(\Delta^n))\otimes\DGN^*(\Delta^n)$,
which we form by composing $\phi = \phi_f$ with the unit of our adjunction relation $\eta$,
corresponds under our bijections (1-3)
to the morphism $\phi_{f^{\sharp}}: \ESym(K\otimes\DGN_*(\Delta^n))\rightarrow\ESym(L\otimes\DGN_*(\Delta^n))$
given by the construction of~\S\ref{Background:CosimplicialFramingConstruction},
and we use the Yoneda lemma to conclude that this morphism $\phi_{f^{\sharp}}$
represents the adjoint of our morphism $\phi = \phi_f$.
\end{proof}

The collection of free $\EOp$-algebras $A\otimes\Delta^n = \ESym(K\otimes\DGN_*(\Delta^n))$, $n\in\NN$,
which we associate to any $A = \ESym(K)\in\FCat({}_{\EOp}\dg\kk)$
clearly defines a cosimplicial object $A\otimes\Delta^{\bullet}$
in the category of $\EOp$-algebras ${}_{\EOp}\dg\kk$.
The result of Proposition~\ref{Background:CosimplicialFramingAdjunction}
also implies that the mapping which associates this cosimplicial object $A\otimes\Delta^{\bullet}$
to any $A = \ESym(K)\in\FCat({}_{\EOp}\dg\kk)$
defines a functor on the category $\FCat({}_{\EOp}\dg\kk)\subset{}_{\EOp}\dg\kk$.
We then have the following observation:

\begin{prop}\label{Background:CosimplicialFraming}
Let $A = \ESym(K)\in\FCat({}_{\EOp}\dg\kk)$. The free $\EOp$-algebras $A\otimes\Delta^n = \ESym(K\otimes\DGN_*(\Delta^n))$, $n\in\NN$,
define a cosimplicial frame of the object $A = \ESym(K)$ in the category of $\EOp$-algebras ${}_{\EOp}\dg\kk$.
\end{prop}

\begin{proof}
We use standard model category arguments to get, through the adjunction relation of Proposition~\ref{Background:CosimplicialFramingAdjunction},
an equivalence between the properties of a cosimplicial frame for the objects $A\otimes\Delta^n$, $n\in\NN$,
and the properties of a simplicial frame for the objects $B^{\Delta^n}$, $n\in\NN$.
We then use the statement of Proposition~\ref{Background:SimplicialFraming}
to conclude.
\end{proof}

We may actually use the adjunction relation of Proposition~\ref{Background:CosimplicialFramingAdjunction}
to extend our functors $-\otimes\Delta^n$ to the whole category of $\EOp$-algebras.
We then check that this construction returns a cosimplicial frame $A\otimes\Delta^{\bullet}$
as soon as $A$ is a cofibrant in~${}_{\EOp}\dg\kk$.

\section{The main result}\label{MainResult}
In what follows, we use the notation $\simp\ACat$ (respectively, $\cosimp\ACat$)
for the category of simplicial (respectively, cosimplicial) objects
in any ambient category $\ACat$.

The geometric realization of a (Reedy cofibrant) simplicial object $A_{\bullet}$ in the category ${}_{\EOp}\dg\kk$
is defined by the following coend:
\begin{equation*}
|A_{\bullet}| = \int^{\underline{n}\in\Delta} A_n\otimes\Delta^n,
\end{equation*}
which runs over the simplicial category $\Delta$, and where we take a cosimplicial frame $A_{\bullet}\otimes\Delta^{\bullet}$
of our object~$A_{\bullet}$ in the category~$\simp({}_{\EOp}\dg\kk)$
equipped with the Reedy model structure of simplicial objects in ${}_{\EOp}\dg\kk$.
Let us insist that we form this coend in the category of $\EOp$-algebras ${}_{\EOp}\dg\kk$.
The geometric realization $|A_{\bullet}|$ automatically forms a cofibrant
object in ${}_{\EOp}\dg\kk$ when $A_{\bullet}$
is Reedy cofibrant.

The cotriple resolution $R_{\bullet} = \Res^{\EOp}_{\bullet}(A)$ which we consider in this paper has this feature.
This simplicial object is equipped with an augmentation $\epsilon: \Res^{\EOp}_{\bullet}(A)\rightarrow A$
and the main purpose of this section is to check that this augmentation morphism
induces a weak-equivalence when we pass
to the geometric realization $\epsilon: \Res^{\EOp}_{\bullet}(A)\xrightarrow{\sim} A$.
To get this result, we use the explicit framing of the previous section and the theory of modules over operads.
We give a brief reminder on the definition of the cotriple resolution
before tackling this proof.

\begin{rem}[The cotriple resolution of algebras over the Barratt-Eccles operad]\label{MainResult:CotripleResolution}
The cotriple resolution $\Res^{\EOp}_{\bullet}(A)\in\simp({}_{\EOp}\dg\kk)$
of an $\EOp$-algebra $A\in{}_{\EOp}\dg\kk$
is defined by:
\begin{equation}\tag{$*$}\label{MainResult:CotripleResolution:Components}
\Res^{\EOp}_n(A) = \underbrace{\ESym\circ\cdots\circ\ESym}_{n+1}(A),
\end{equation}
for each dimension $n\in\NN$, where we consider an $n+1$-fold composite of free $\EOp$-algebra
functors (and forgetful functors which we omit to mark for simplicity).
We number the factors of this composite by $0,\dots,n$, from left to right.
We use the free $\EOp$-algebra adjunction $\ESym(-): \dg\kk\rightleftarrows{}_{\EOp}\dg\kk :\omega$
to define the structure morphisms of this simplicial object.
To be explicit, we equip $\Res^{\EOp}_{\bullet}(A)$ with the face morphisms $d_i: \Res^{\EOp}_n(A)\rightarrow\Res^{\EOp}_{n-1}(A)$
given, for $i = 0,\dots,n$, by the application of the augmentation morphism of our adjunction $\epsilon: \ESym(-) = \ESym\omega(-)\rightarrow\Id$
to the $i$th factor of our composite functor~(\ref{MainResult:CotripleResolution:Components}),
while the degeneracy morphisms $s_j: \Res^{\EOp}_n(A)\rightarrow\Res^{\EOp}_{n+1}(A)$,
are given, for $j = 0,\dots,n$, by the insertion of the adjunction unit $\iota: \Id\rightarrow\omega\ESym(-)$
between the $j$th and $j+1$th factors of this composite.

This simplicial object forms a free $\EOp$-algebra dimension-wise. We explicitly have $\Res^{\EOp}_n(A) = \ESym(\DGC^{\EOp}_n(A))$, for any $n\in\NN$,
for a generating dg-module $\DGC^{\EOp}_n(A)\in\dg\kk$
such that:
\begin{equation*}
\DGC^{\EOp}_n(A) = \underbrace{\ESym\circ\cdots\circ\ESym}_{n}(A).
\end{equation*}
We immediately see that the face operators $d_i$ such that $i>0$
are identified with morphisms
of free $\EOp$-algebras $d_i: \ESym(\DGC^{\EOp}_n(A))\rightarrow\ESym(\DGC^{\EOp}_{n-1}(A))$
which we associate to face morphisms of these generating dg-modules $d_i: \DGC^{\EOp}_n(A)\rightarrow\DGC^{\EOp}_{n-1}(A)$.
We have a similar relation for the degeneracy operators $s_j: \Res^{\EOp}_n(A)\rightarrow\Res^{\EOp}_{n+1}(A)$, for all $j$.
The zeroth face $d_0: \Res^{\EOp}_n(A)\rightarrow\Res^{\EOp}_{n-1}(A)$, on the other hand,
is yielded by a morphism $d_0: \ESym(\DGC^{\EOp}_n(A))\rightarrow\ESym(\DGC^{\EOp}_{n-1}(A))$
which does not preserve our generating dg-modules.

We also equip $\Res^{\EOp}_{\bullet}(A)\in\simp({}_{\EOp}\dg\kk)$
with an augmentation
$\Res^{\EOp}_0(A) = \ESym(A)\xrightarrow{\epsilon} A$
which is yielded by the augmentation morphism of the free $\EOp$-algebra adjunction.
\end{rem}

We record the following statement:

\begin{prop}\label{MainResult:ReedyCofibrant}
The simplicial object $\Res^{\EOp}_{\bullet}(A)\in\simp({}_{\EOp}\dg\kk)$
is Reedy cofibrant
as soon as the $\EOp$-algebra $A$
is cofibrant as a dg-module (which is automatically the case since we assume that our ground ring is a field).
\end{prop}

\begin{proof}
We regard the dg-modules $\DGC^{\EOp}_n(A)$, $n\in\NN$, which generate the terms of the cotriple resolution $\Res^{\EOp}_n(A)$, $n\in\NN$,
as the components of a simplicial object equipped with a trivial zeroth face operator $d_0 = 0$.

We use that the elements of the dg-module $\DGC^{\EOp}_n(A)$ can be represented, for any $n\in\NN$,
by tensors arranged on trees with $n$ levels, numbered from $1$ to $n$.
We basically label the vertices of these trees $v$ by elements of the operad $\pi_v\in\EOp(r_v)$,
and we label the leaves by elements of our algebra $A$ (see~\cite[\S 4.3]{PartitionHomology}).
We mainly require that the inputs of the operad elements $\pi_v\in\EOp(r_v)$
are in bijection with the ingoing edges of the corresponding vertex $v$.
Let $\DGN_n\DGC^{\EOp}_{\bullet}(A)\subset\DGC^{\EOp}_n(A)$ be the submodule of $\DGC^{\EOp}_n(A)$
spanned by these tree-wise tensors where we have no level which entirely consists of vertices
with one ingoing edge (and labeled by the operad unit $1\in\EOp(1)$ therefore).
This module $\DGN_n\DGC^{\EOp}_{\bullet}(A)$ represents a section of the quotient of the module $\DGC^{\EOp}_n(A)$
over the image of the degeneracy operators.

We now have a decomposition
\begin{equation*}
\DGC^{\EOp}_n(A) = \DGL_n\DGC^{\EOp}_{\bullet}(A)\oplus\DGN_n\DGC^{\EOp}_{\bullet}(A),
\end{equation*}
where we consider the $n$th latching object $\DGL_n\DGC^{\EOp}_{\bullet}(A)$
of this simplicial object $\DGC^{\EOp}_{\bullet}(A)$
in the category of dg-modules.
We have on the other hand $\DGL_n\Res^{\EOp}_{\bullet}(A) = \ESym(\DGL_n\DGC^{\EOp}_{\bullet}(A))$,
where we consider the $n$th latching object $\DGL_n\Res^{\EOp}_n(A)$ of the cotriple resolution $\Res^{\EOp}_{\bullet}(A)$,
because the degeneracy operators, which we use in the definition
of this latching object, are identified with morphisms
of free $\EOp$-algebras
induced by the degeneracy operators of $\DGC^{\EOp}_{\bullet}(A)$.
We accordingly have the identity:
\begin{equation*}
\Res^{\EOp}_n(A) = \DGL_n\Res^{\EOp}_{\bullet}(A)\vee\ESym(\DGN_n\DGC^{\EOp}_{\bullet}(A)),
\end{equation*}
in any dimension $n\in\NN$, where $\vee$ denote the coproduct in the category of $\EOp$-algebras.
The latching morphism $\lambda: \DGL_n\Res^{\EOp}_{\bullet}(A)\rightarrow\Res^{\EOp}_n(A)$ which we associate to this latching object
is identified with the canonical inclusion of the summand $\DGL_n\Res^{\EOp}_{\bullet}(A)$
in this coproduct with the free $\EOp$-algebra~$\ESym(\DGN_n\DGC^{\EOp}_{\bullet}(A))$,
which forms a cofibrant object by definition of the model structure
on the category of $\EOp$-algebras.
We conclude from this examination that this latching morphism defines a cofibration, for each $n\in\NN$,
and hence, that the cotriple resolution $\Res^{\EOp}_{\bullet}(A)$
is Reedy cofibrant as a simplicial object
in the model category of $\EOp$-algebras.
\end{proof}

We apply the cosimplicial framing construction of~\S\ref{Background:CosimplicialFramingConstruction} to the free $\EOp$-algebras $\Res^{\EOp}_n(A)$,
and we use this particular cosimplicial framing to determine the geometric realization
of the cotriple resolution $\Res^{\EOp}_{\bullet}(A)$.
We accordingly set:
\begin{equation*}
|\Res^{\EOp}_{\bullet}(A)| = \int^{\underline{n}\in\Delta}\ESym(\underbrace{\ESym\circ\dots\circ\ESym}_n(A)\otimes\DGN_*(\Delta^n)),
\end{equation*}
for any $A\in{}_{\EOp}\dg\kk$.
We also have the following structure result:

\begin{prop}\label{MainResult:QuasiFreeStructure}
The geometric realization of the cotriple resolution is identified with a quasi-free $\EOp$-algebra (in the sense of~\cite[\S 12.3.6]{OperadModules}).
We explicitly have an identity:
\begin{equation*}
|\Res^{\EOp}_{\bullet}(A)| = (\ESym(\DGN_*\DGC_{\bullet}(A)),\partial),
\end{equation*}
where we consider the normalized complex $\DGN_*\DGC_{\bullet}(A)$ of the simplicial dg-module $\DGC_{\bullet}(A)$
equipped with a trivial zeroth face $d_0 = 0$,
and where we have a twisting derivation
$\partial: \ESym(\DGN_*\DGC_{\bullet}(A))\rightarrow\ESym(\DGN_*\DGC_{\bullet}(A))$
which we add to the natural differential of the free $\EOp$-algebra $\ESym(\DGN_*\DGC_{\bullet}(A))$
in order to form the differential of our quasi-free object. This twisting derivation $\partial$ maps the elements
of the generating dg-module $\DGN_*\DGC_{\bullet}(A)\subset\ESym(\DGN_*\DGC_{\bullet}(A))$
to decomposable elements of the free $\EOp$-algebra $\ESym(\DGN_*\DGC_{\bullet}(A))$.
\end{prop}

\begin{proof}
We use the definition of the normalized complex $\DGN_*(\Delta^n)$ as the quotient
of the Moore complex $\DGC_*(\Delta^n) = \kk[\Delta^n]$
over the image of the degeneracy operators.
We adopt the notation $[\sigma]\in\DGN_m(\Delta^k)$ for the class of a simplex $\sigma\in(\Delta^k)_m$ in $\DGN_*(\Delta^k)$.
We use similar conventions for the normalized complex of the simplicial object of the category of dg-modules $\DGC_{\bullet}(A)$.
We can identify our coend, in the definition of the geometric realization $|\Res^{\EOp}_{\bullet}(A)|$,
with the quotient of the free $\EOp$-algebra $F = \ESym(\bigoplus_n\DGC_n(A)\otimes\DGN_*(\Delta^n))$
over the ideal generated by the relations:
\begin{equation}\label{MainResult:QuasiFreeStructure:GeneratingRelations}
\xi\otimes u_*([\sigma])\equiv(u^*\otimes\Delta^k)(\xi\otimes[\sigma]),
\end{equation}
for any tensor $\xi\otimes[\sigma]\in\DGC_n(A)\otimes\DGN_m(\Delta^k)$ and any morphism $u: \underline{k}\rightarrow\underline{n}$
of the simplicial category $\Delta$,
where we consider the generating element $\xi\otimes u_*([\sigma])\in\DGC_n(A)\otimes\DGN_*(\Delta^n)$
on the left-hand side,
and the image of the element $\xi\otimes[\sigma]\in\DGC_n(A)\otimes\DGN_m(\Delta^n)$
under the morphism of cosimplicial frames
\begin{equation}\label{MainResult:QuasiFreeStructure:SimplicialOperators}
\ESym(\DGC_n(A)\otimes\DGN_*(\Delta^k))\xrightarrow{u^*\otimes\Delta^k}\ESym(\DGC_k(A)\otimes\DGN_*(\Delta^k))
\subset F
\end{equation}
associated to the simplicial operator $u^*: \Res^{\EOp}_n(A)\rightarrow\Res^{\EOp}_k(A)$
of the cotriple resolution $\Res^{\EOp}_{\bullet}(A) = \ESym(\DGC_{\bullet}(A))$
on the right-hand side.

We forget about differentials for the moment.
We use the subscript $\flat$ to mark this forgetful operation.
We consider the morphism of $\EOp_{\flat}$-algebras
\begin{equation}\label{MainResult:QuasiFreeStructure:Map}
\ESym(\DGN_*\DGC_{\bullet}(A))_{\flat}\xrightarrow{\phi}\int^{\underline{k}\in\Delta}\ESym(\DGC_k(A)\otimes\DGN_*(\Delta^k))_{\flat}
= |\Res^{\EOp}_{\bullet}(A)|_{\flat}
\end{equation}
which carries any generating element $[\xi]\in\DGN_n\DGC_{\bullet}(A)$ to the tensor $\xi\otimes[i_n]\in\DGC_n(A)\otimes\DGN_n(\Delta^n)$,
where $i_n$ denotes the fundamental simplex of the simplicial set $\Delta^n$.
We simply note that we have $s_j(\xi)\otimes[i_n]\equiv\xi\otimes[s^j(i_n)]$
according to our coend relations, and hence, we get $[s^j(i_n)] = [s_j(i_n)]\equiv 0\Rightarrow s_j(\xi)\otimes[i_n]\equiv 0$,
from which we conclude that our morphism~(\ref{MainResult:QuasiFreeStructure:Map})
is well defined.
We aim to prove that this map~(\ref{MainResult:QuasiFreeStructure:Map})
is a bijection.
We use a reduction procedure to define a morphism $\psi$
such that $\psi\phi = \id$ and $\phi\psi = \id$.

Let $\xi\otimes[\sigma]\in\DGC_n(A)\otimes\DGN_m(\Delta^n)$
be a generating element
of our coend.
We have $\xi\otimes[\sigma]\equiv(\sigma^*\otimes\Delta^m)(\xi\otimes[i_m])$,
where we consider the morphism $\sigma: \underline{m}\rightarrow\underline{n}$ and the simplicial operator $\sigma^*$
associated to the simplex $\sigma\in(\Delta^n)_m$. We use that the image of $[i_m]\in\DGN_m(\Delta^m)$
under the coaction of the Barratt-Eccles operad
on $\DGN_*(\Delta^m)$
involves factors $[i_m]\in\DGN_m(\Delta^m)$ and simplices $[\tau]\in\DGN_k(\Delta^m)$
of degree $k<m$.
We can therefore use an induction process to assign an element $\psi(\xi\otimes[\sigma]) := \psi(\sigma^*\otimes\Delta^m)(\xi\otimes[i_m])\in\ESym(\DGN_*\DGC_{\bullet}(A))$
to any $\xi\otimes[\sigma]\in\DGC_n(A)\otimes\DGN_m(\Delta^n)$.

We just have to check that this construction carries each side of our generating relations~(\ref{MainResult:QuasiFreeStructure:GeneratingRelations})
to the same element in $\ESym(\DGN_*\DGC_{\bullet}(A))_{\flat}$.
We proceed by induction on the simplicial dimension $n\in\NN$ of these relations.
We use that the preservation of the generating relations in dimension $\leq n-1$,
implies that our mapping yields a well-defined morphism
on the $n-1$ dimensional skeleton of our coend:
\begin{equation}\label{MainResult:QuasiFreeStructure:ConverseMap}
\sk_{n-1}|\Res^{\EOp}_{\bullet}(A)|_{\flat} := \int^{\underline{k}\in\Delta_{\leq n-1}}\ESym(\DGC_k(A)\otimes\DGN_*(\Delta^k))_{\flat}
\xrightarrow{\psi}\ESym(\DGN_*\DGC_{\bullet}(A))_{\flat}.
\end{equation}
We equivalently use that all relations of our coend are satisfied in dimension $\leq n-1$
when this is so for generating elements.

We fix a tensor a tensor $\xi\otimes[\sigma]\in\DGC_n(A)\otimes\DGN_m(\Delta^k))$
and a morphism $u: \underline{k}\rightarrow\underline{n}$
as in our relation~(\ref{MainResult:QuasiFreeStructure:GeneratingRelations}).
We have $(u^*\otimes\Delta^k)(\xi\otimes[\sigma]) = (u^*\otimes\Delta^k)(\xi\otimes\sigma_*[i_m])
= (\id\otimes\sigma_*)(u^*\otimes\Delta^m)(\xi\otimes[i_m])$,
where we consider the cosimplicial structure morphism
$\id\otimes\sigma_*: \ESym(\DGC_n(A)\otimes\DGN_*(\Delta^m))\rightarrow\ESym(\DGC_n(A)\otimes\DGN_*(\Delta^k))$
of the cosimplicial frame $\ESym(\DGC_n(A)\otimes\DGN_*(\Delta^{\bullet}))$
associated to the map $\sigma: \underline{m}\rightarrow\underline{k}$.
We see that this element $\varpi = (\id\otimes\sigma_*)(u^*\otimes\Delta^m)(\xi\otimes[i_m])$
is given by applying the map $\sigma_*: \DGN_*(\Delta^m)\rightarrow\DGN_*(\Delta^k)$
to the factors $[\tau]\in\DGN_*(\Delta^m)$
occurring in the expansion of $(u^*\otimes\Delta^m)(\xi\otimes[i_m])\in\ESym(\DGC_k(A)\otimes\DGN_*(\Delta^m))$.
We accordingly have the relation $\psi((\id\otimes\sigma_*)(u^*\otimes\Delta^m)(\xi\otimes[i_m]))
= \psi((\sigma^*\otimes\Delta^m)(u^*\otimes\Delta^m)(\xi\otimes[i_m]))$
by induction hypothesis,
from which we conclude that we have $\psi((u^*\otimes\Delta^k)(\xi\otimes[\sigma]) = \psi((u\sigma)^*\otimes\Delta^m)(\xi\otimes[i_m])$,
and this verification finishes the proof that our morphism $\psi$
is well defined.
We simply note that we clearly have $\psi\phi = \id$ and $\phi\psi = \id$ by definition of this map $\psi$
to complete the proof that our morphism~(\ref{MainResult:QuasiFreeStructure:Map})
is a bijection.

We use this bijection to transport the differential of the $\EOp$-algebra $|\Res^{\EOp}_{\bullet}(A)|$
to the object $\ESym(\DGN_*\DGC_{\bullet}(A))$.
We define the twisting derivation of the proposition as the difference between this differential inherited from $|\Res^{\EOp}_{\bullet}(A)|$
and the natural differential of the free $\EOp$-algebra $\ESym(\DGN_*\DGC_{\bullet}(A))$.
For a generating element $[\xi]\in\DGN_n\DGC_{\bullet}(A)$,
we have the formula:
\begin{multline}\label{MainResult:QuasiFreeStructure:Differential}
\textstyle{\delta\phi[\xi] = \delta(\xi\otimes[i_n]) = \delta(\xi)\otimes[i_n] + \pm\sum_{i=0}^n (-1)^i\xi\otimes d_i[i_n]} \\
\textstyle{= \delta(\xi)\otimes[i_n] + \pm\xi\otimes[d_0(i_{n-1})] + \pm\sum\nolimits_{i=1}^n (-1)^i d_i(\xi)\otimes[i_{n-1}]},
\end{multline}
from which we readily conclude that $\partial[\xi]$ is identified with the image
of the tensor $\xi\otimes[d_0(i_{n-1})]\in\DGC_n(A)\otimes\DGN_{n-1}(\Delta^n)$
under our reduction procedure.
Hence, we get:
\begin{equation}
\partial[\xi] = \psi((d_0\otimes\Delta^{n-1})(\xi\otimes[i_{n-1}])).
\end{equation}
If we use the representation of the complex $\DGN_*\DGC_{\bullet}(A)$ given in the proof of Proposition~\ref{MainResult:ReedyCofibrant},
then we readily see that $d_0(\xi)$
defines a decomposable element of the free $\EOp$-algebra.
This observation implies that $\partial[\xi]$ is decomposable yet, because we inductively apply our map $\psi$
to the factors of this tensor $d_0(\xi)\in\ESym(\DGC_{n-1}(A))$
in order to get our result.
The proof of our proposition is now complete.
\end{proof}

\begin{constr}[The applications of modules over operads]\label{MainResult:OperadModules}
We can observe that the functor $|\Res^{\EOp}_{\bullet}(-)|$
is represented by a bimodule over the operad~$\EOp$
in the sense of~\cite[\S 9.2]{OperadModules}.
We give a brief survey of this correspondence between modules over operads and functors with a view towards the applications
which we consider in this paper. We refer to the cited monograph~\cite{OperadModules}
for details on this subject.

We first consider the category of right $\EOp$-modules $\MCat{}_{\EOp}$,
whose objects are symmetric sequences $\MOp$
equipped with a right action $\rho: \MOp\circ\EOp\rightarrow\MOp$
of the operad $\EOp$.
To any such object $\MOp\in\MCat{}_{\EOp}$, we associate a functor $\SSym_{\EOp}(\MOp,-): {}_{\EOp}\dg\kk\rightarrow\dg\kk$,
which we can define by a relative composition product:
\begin{equation*}
\SSym_{\EOp}(\MOp,A) = \MOp\circ_{\EOp} A,
\end{equation*}
where we identify the $\EOp$-algebra $A\in{}_{\EOp}\dg\kk$
with a symmetric sequence concentrated in arity $0$
and equipped with a left action of the operad $\EOp$ (a left $\EOp$-module concentrated in arity $0$).
The category $\MCat{}_{\EOp}$ is equipped with a symmetric monoidal structure
that reflects the point-wise tensor product of functors
on $\EOp$-algebras. We moreover have an obvious symmetric monoidal functor $\eta: \dg\kk\rightarrow\MCat{}_{\EOp}$
which identified any dg-module $K\in\dg\kk$
with the symmetric sequence $\MOp$ such that $\MOp(0) = K$ and $\MOp(r) = 0$ for $r>0$.
The $\EOp$-bimodules which we use in our construction are identified with $\EOp$-algebras in the symmetric monoidal category $\MCat{}_{\EOp}$,
and are associated to endofunctor of the category of $\EOp$-algebras (see~\cite[\S 9.1]{OperadModules}).
We also use the notation ${}_{\EOp}\MCat{}_{\EOp}$
for this category of $\EOp$-bimodules. The operad $\EOp$ naturally forms a bimodule over itself.
The categories $\MCat_{\EOp}$ and ${}_{\EOp}\MCat{}_{\EOp}$ also inherit a natural model structure
from the base category of dg-modules.

We can readily extend the constructions of the previous paragraphs to $\EOp$-algebras in right $\EOp$-modules.
We basically replace the operations on the category of dg-modules which we use in our constructions
by structure operations of the symmetric monoidal category of right $\EOp$-modules $\MCat{}_{\EOp}$.
We now consider the geometric realization of the cotriple resolution of the operad $\EOp$,
which we regard as a bimodule over itself or, equivalently, as an $\EOp$-algebra
in the category of right $\EOp$-modules.
This construction returns an $\EOp$-algebra in right $\EOp$-modules $|\Res^{\EOp}_{\bullet}(\EOp)|\in{}_{\EOp}\MCat{}_{\EOp}$,
and, according to the general statement of~\cite[\S 9.2.6]{OperadModules},
we have the relation:
\begin{equation*}
|\Res^{\EOp}_{\bullet}(A)| = |\Res^{\EOp}_{\bullet}(\EOp)|\circ_{\EOp} A,
\end{equation*}
for any $A\in{}_{\EOp}\dg\kk$.
The augmentation of the cotriple resolution~$\epsilon: |\Res^{\EOp}_{\bullet}(A)|\rightarrow A$
is similarly yielded by a morphism
of $\EOp$-bimodules
$\epsilon: |\Res^{\EOp}_{\bullet}(\EOp)|\rightarrow\EOp$,
which we determine from an augmentation of the simplicial object $\Res^{\EOp}_{\bullet}(\EOp)$.
\end{constr}

Let us mention that our proof of Proposition~\ref{MainResult:ReedyCofibrant} remains valid in the category of right $\EOp$-modules
(since we essentially use operations of the underlying symmetric monoidal category in our constructions),
and that $\Res^{\EOp}_{\bullet}(\EOp)$
is therefore Reedy cofibrant
as a simplicial object in the model category of $\EOp$-algebras in right $\EOp$-modules.
The idea is to establish our weak-equivalence statement for the object $|\Res^{\EOp}_{\bullet}(\EOp)|\in{}_{\EOp}\MCat{}_{\EOp}$
and to use general homotopy invariance theorems of~\cite[\S 15]{OperadModules}
in order to deduce the case of $\EOp$-algebras in dg-modules
of our statement
from this universal case.
To be explicit, we gain ou main result from the following statement:

\begin{lemm}\label{MainResult:ModuleEquivalence}
The augmentation of the simplicial object $R_{\bullet} = \Res^{\EOp}_{\bullet}(\EOp)$
induces a weak-equivalence $\epsilon: |\Res^{\EOp}_{\bullet}(\EOp)|\xrightarrow{\sim}\EOp$
when we pass to the geometric realization.
\end{lemm}

\begin{proof}
The quasi-free object decomposition of Proposition~\ref{MainResult:QuasiFreeStructure}
remains valid when we work in the category of right $\EOp$-modules (like the result of Proposition~\ref{MainResult:ReedyCofibrant}).
We then have
\begin{equation*}
|\Res^{\EOp}_{\bullet}(\EOp)| = (\ESym(\DGN_*\DGC_{\bullet}(\EOp)),\partial),
\end{equation*}
where we take the free $\EOp$-algebra on the right $\EOp$-module $\DGN_*\DGC_{\bullet}(\EOp)$
and where $\partial: \ESym(\DGN_*\DGC_{\bullet}(\EOp))\rightarrow\ESym(\DGN_*\DGC_{\bullet}(\EOp))$
is now a derivation of $\EOp$-algebras in right $\EOp$-modules.

We filter the object $\ESym(\DGN_*\DGC_{\bullet}(\EOp))$ by the weight grading of the free $\EOp$-algebra $\ESym(-)$.
We immediately see that this filtration is bounded in each arity $n>0$, when we fix a term $\ESym(\DGN_*\DGC_{\bullet}(\EOp))(n)$
of the right $\EOp$-module underlying $\ESym(\DGN_*\DGC_{\bullet}(\EOp))$,
because the tensor powers of a right $\EOp$-module $\MOp$
satisfy $\MOp^{\otimes r}(n) = 0$ for $n>r$
when we have $\MOp(0) = 0$.
The twisting derivation $\partial: \ESym(\DGN_*\DGC_{\bullet}(\EOp))\rightarrow\ESym(\DGN_*\DGC_{\bullet}(\EOp))$ increases this weight filtration,
because the observation that $\partial$ carries generators to decomposable elements in Proposition~\ref{MainResult:QuasiFreeStructure}
implies that this increasing statement is fulfilled for generators.
We accordingly have a convergent spectral sequence $\DGE^r(\Res^{\EOp}_{\bullet}(\EOp))\Rightarrow\DGH_*|\Res^{\EOp}_{\bullet}(\EOp)|$
such that $\DGE^0(\Res^{\EOp}_{\bullet}(\EOp)) = \ESym(\DGN_*\DGC_{\bullet}(\EOp))$.
We also have an identity $\ESym(\DGN_*\DGC_{\bullet}(\EOp)) = \EOp\circ\DGN_*\DGC_{\bullet}(\EOp)$
when we use the equivalence between $\EOp$-algebras in right $\EOp$-modules
and $\EOp$-bimodules,
and the operadic K\"unneth formula $\DGH_*(\MOp\circ\NOp) = \DGH_*(\MOp)\circ\DGH_*(\NOp)$ (see~\cite[Lemma 2.1.3]{PartitionHomology})
implies that we have the identity $\DGE^1(\Res^{\EOp}_{\bullet}(\EOp)) = \DGH_*(\EOp)\circ\DGH_*\DGN_*\DGC_{\bullet}(\EOp)$
on the $E^1$-page of our spectral sequence $\DGE^r(\Res^{\EOp}_{\bullet}(\EOp))$.

We can identify the right $\EOp$-module $\DGC_{\bullet}(\EOp)$ with the instance $C_{\bullet} = C_{\bullet}(\IOp,\EOp,\EOp)$
of the simplicial bar complex of~\cite[\S 4.1.5]{PartitionHomology}, where $\IOp$
denotes the unit operad.
The results of~\cite[Lemma 4.7.4 and Theorem 4.1.8]{PartitionHomology} imply that have an identity $\DGH_*\DGN_*\DGC_{\bullet}(\EOp) = \IOp$,
from which we readily deduce that our augmentation morphism $\epsilon: |\Res^{\EOp}_{\bullet}(\EOp)|\rightarrow\EOp$
induces an isomorphism on the $E^1$-page
of our spectral sequence.
Then we can use standard comparison arguments to get our conclusion.
\end{proof}

We can now establish the main statement of this paper:

\begin{thm}\label{MainResult:Statement}
The geometric realization of the cotriple resolution of an $\EOp$-algebra in dg-modules $A\in{}_{\EOp}\dg\kk$
is endowed with a weak-equivalence $\epsilon: |\Res^{\EOp}_{\bullet}(A)|\xrightarrow{\sim} A$
and defines a cofibrant resolution of our object $A$
in the model category of $\EOp$-algebras ${}_{\EOp}\dg\kk$.
\end{thm}

Recall that we assume that the ground ring is field all through this paper, and that this assumption implies that all objects
are cofibrant in the category of dg-modules.
The above theorem is actually valid without assuming that the ground ring is field,
but we need to restrict ourselves to $\EOp$-algebras which are cofibrant as dg-modules in this case.

\begin{proof}
We already mentioned that the result of Proposition~\ref{MainResult:ReedyCofibrant} holds for the object $\Res^{\EOp}_{\bullet}(\EOp)$
of which geometric realization $|\Res^{\EOp}_{\bullet}(\EOp)|$ therefore forms a cofibrant object
in the category of $\EOp$-algebras in right $\EOp$-modules.
By the general observation of~\cite[Proposition 12.3.2]{OperadModules} (where we use the correspondence of~\cite[Proposition 12.3.2]{OperadModules}),
this assertion implies that $|\Res^{\EOp}_{\bullet}(\EOp)|$
is also cofibrant as a right $\EOp$-module.
The operad $\EOp$ also trivially forms a cofibrant object of the category of right $\EOp$-modules
by definition of our model structure
on this category (see~\cite[Proposition 14.1.A]{OperadModules}).
In this situation, the weak-equivalence $\epsilon: |\Res^{\EOp}_{\bullet}(\EOp)|\xrightarrow{\sim}\EOp$
induces a weak-equivalence
\begin{equation*}
|\Res^{\EOp}_{\bullet}(\EOp)|\circ_{\EOp}A\xrightarrow{\sim}\EOp\circ_{\EOp}A = A,
\end{equation*}
when we pass to the relative composition product, for any $\EOp$-algebra in dg-modules $A\in{}_{\EOp}\dg\kk$
(see~\cite[Theorem 15.1.A(a)]{OperadModules}).
This result gives the conclusion of the theorem since the above morphism is identified with the augmentation
of the $\EOp$-algebra $|\Res^{\EOp}_{\bullet}(A)| = |\Res^{\EOp}_{\bullet}(\EOp)|\circ_{\EOp}A$.
\end{proof}

\section{The unitary setting}\label{UnitarySetting}

We can adapt the constructions of the previous sections to algebras over the unitary version of the Barratt-Eccles operad $\EOp_+$
which has $\EOp_+(0) = \kk$ as arity zero term.

We use the same definition of cosimplicial framing in the category ${}_{\EOp_+}\dg\kk$ as in the case
of non-unitary $\EOp$-algebras~(\S\ref{Background:CosimplicialFramingConstruction}).
We just keep track of the extra term $\EOp_+(0) = \kk$
in the free $\EOp_+$-algebra $\ESym_+(K) = \bigoplus_{r=0}^{\infty}(\EOp_+(r)\otimes K^{\otimes r})_{\Sigma_r}$.
We also have $\ESym_+(K) = \kk\oplus\ESym(K)$ since $\EOp_+(r) = \EOp(r)$ for $r>0$.
Each $\EOp_+$-algebra $A$ inherits a unit morphism $\eta: \kk\rightarrow A$
which is given by the identity between the ground field
and the term of arity zero operations
of our operad $\EOp_+(0) = \kk$. We consider the functor $\omega: {}_{\EOp_+}\dg\kk\rightarrow\kk/\dg\kk$
which retains this unit morphism.
We can form a relative coproduct $\kk/\ESym_+(K) = \kk\bigvee_{\ESym_+(\kk)}\ESym_+(K)$ for any object $K\in\kk/\dg\kk$
to define a left adjoint $\kk/\ESym_+(-): \kk/\dg\kk\rightarrow{}_{\EOp_+}\dg\kk$
of this partial forgetful functor $\omega: {}_{\EOp_+}\dg\kk\rightarrow\kk/\dg\kk$.
We now consider the cotriple resolution functor
\begin{equation*}
\Res^{\EOp_+}_{\bullet}(-): {}_{\EOp_+}\dg\kk\rightarrow\simp({}_{\EOp_+}\dg\kk)
\end{equation*}
which we deduce from this adjoint pair~$\kk/\ESym_+(-): \kk/\dg\kk\rightleftarrows{}_{\EOp_+}\dg\kk :\omega$.
To any $\EOp$-algebra $A$, we can also associate an augmented $\EOp_+$-algebras $A_+$ such that $A_+ = \kk\oplus A$.
This construction clearly defines an isomorphism of categories between the category of $\EOp$-algebras
and the category of augmented $\EOp_+$-algebras. We actually have $\ESym(K)_+ = \ESym_+(K)$
in the case of a free $\EOp$-algebra $A = \ESym(K)$.
For an object $K_+\in\kk/\dg\kk$ of the form $K_+ = \kk\oplus K$, we have the obvious relation $\kk/\ESym_+(K_+) = \ESym(K)_+$,
where we consider the augmented $\EOp_+$-algebra associated to the free $\EOp$-algebra $\ESym(K)\in{}_{\EOp}\dg\kk$.
By a straightforward induction, we then get:
\begin{equation*}
\Res^{\EOp_+}_n(A_+) = \ESym_+(\underbrace{\ESym\circ\dots\circ\ESym}_n(A)),
\end{equation*}
for any dimension $n\in\NN$, so that we have the relation $\Res^{\EOp_+}_{\bullet}(A_+) = \Res^{\EOp}_{\bullet}(A)_+$
in the category of $\EOp_+$-algebras. The construction which we consider in this section therefore defines
an extension of our cotriple resolution
functor on $\EOp$-algebras.

We can now extend the results of Theorem~\ref{MainResult:Statement} to the geometric realization
of the simplicial object $\Res^{\EOp_+}_{\bullet}(A)$
which we associate to any $A\in{}_{\EOp_+}\dg\kk$.
To be explicit, we have the following statement:

\begin{thm}\label{UnitarySetting:MainResult}
The geometric realization of the cotriple resolution of an $\EOp_+$-algebra in dg-modules $A\in{}_{\EOp_+}\dg\kk$
is endowed with a weak-equivalence $\epsilon: |\Res^{\EOp_+}_{\bullet}(A)|\xrightarrow{\sim} A$
and defines a cofibrant resolution of our object $A$
in the model category of $\EOp_+$-algebras ${}_{\EOp_+}\dg\kk$
as soon as $A\not=0$.
\end{thm}

Recall that we assume that the ground ring is a field in general.
The above theorem is actually valid without this assumption (like the result of Theorem~\ref{MainResult:Statement}),
but we need to restrict ourselves to $\EOp_+$-algebras
equipped with a unit morphism $\kk = \EOp_+(0)\rightarrow A$ which is cofibration in the model category of dg-modules
in this case. (If the ground ring is a field, then this requirement is fulfilled as soon as $A\not=0$.)

\begin{proof}[Proof (outline)]
We can adapt the proof of Theorem~\ref{MainResult:Statement} to the case considered in this statement.
Indeed, the object $|\Res^{\EOp_+}_n(\EOp_+)|$ which we consider in our argument line
is no longer cofibrant as a right $\EOp_+$-module
when we consider the projective model structure of~\cite[\S 14.1.A]{OperadModules},
but this object is cofibrant as a symmetric sequence,
and, by~\cite[Theorem 15.2.A(a)]{OperadModules},
this statement is enough to prove that the relative composition product which we form in our argument line
returns a weak-equivalence $|\Res^{\EOp_+}_n(\EOp_+)|\circ_{\EOp_+}A\xrightarrow{\sim}\EOp_+\circ_{\EOp_+}A = A$
when $A$ is cofibrant as an $\EOp_+$-algebra.
Now, we may see the functor $|\Res^{\EOp_+}_n(\EOp_+)|\circ_{\EOp_+}A = |\Res^{\EOp_+}_n(A)|$
carries any weak-equivalence of $\EOp_+$-algebras such that $A\not=0$
to a weak-equivalence.
The general case of our statement can therefore be established from the particular case of cofibrant $\EOp_+$-algebras.
\end{proof}

We can establish the previous theorem by a direct generalization of the argument line of Theorem~\ref{MainResult:Statement},
by using the notion of an augmented $\Lambda$-operad,
which we introduce in~\cite[\S I.3.2]{FresseBook} to model operads with a symmetric monoidal unit
as arity zero term.
The letter $\Lambda$ refers to the category which has the finite ordinals as objects, the injective maps as morphisms,
and which we use to govern the underlying diagram structure of these $\Lambda$-operads.
In the $\Lambda$-operad setting, we can define a model structure on the category of right $\EOp_+$-modules,
which has less fibrations but more cofibrations,
and which we deduce from a Reedy model structure on the category of contravariant $\Lambda$-diagrams (see~\cite[\S II.8.3.4]{FresseBook}).

The object $|\Res^{\EOp_+}_n(\EOp_+)|$ is cofibrant with respect to this Reedy model structure, and one can check that the homotopy invariance statement,
which we use to conclude the proof of Theorem~\ref{MainResult:Statement},
holds for such cofibrant objects
when we can restrict ourselves to $\EOp_+$-algebras
whose unit morphism $\eta: \kk\rightarrow A$
defines a cofibration (which is always true in our setting, unless $A = 0$).
This general homotopy invariance statement reflects the observation
that the functor $|\Res^{\EOp_+}_n(\EOp_+)|\circ_{\EOp_+}A = |\Res^{\EOp_+}_n(A)|$
preserves all weak-equivalences of $\EOp_+$-algebras
such that $A\not=0$.

\section{The case of commutative algebras}\label{CommutativeAlgebraCase}
We also have a counterpart of the result of Theorem~\ref{UnitarySetting:MainResult} for unitary commutative dg-algebras.
Recall that we use the notation $\ComOp$ for the operad associated to the category of (non-unitary) commutative algebras.
This operad satisfies $\ComOp(0) = 0$ and $\ComOp(r) = \kk$ for $r>0$.
We identify the category of unitary commutative dg-algebras with the category of algebras ${}_{\ComOp_+}\dg\kk$
associated to the operad $\ComOp_+$ such that $\ComOp_+(0) = \kk$ and $\ComOp_+(r) = \ComOp(r)$ for $r>0$.
We use that the symmetric algebra $\SSym(K) = \bigoplus_{r=0}^{\infty}(K^{\otimes r})_{\Sigma_r}$
defines a free object functor with values in the category of unitary commutative dg-algebras ${}_{\ComOp_+}\dg\kk$.
We also use a relative tensor product construction $\kk/\SSym(K) = \kk\otimes_{\SSym(\kk)}\SSym(K)$
to define a reduced version of this symmetric algebra construction
for any dg-module $K\in\kk/\dg\kk$
endowed with a morphism $\eta: \kk\rightarrow K$.
We still immediately see that the mapping $\kk/\SSym(-): K\mapsto\kk/\SSym(K)$ defines a left adjoint
of the obvious functor $\omega: {}_{\ComOp_+}\dg\kk\rightarrow\kk/\dg\kk$.
We then consider the cotriple resolution $\Res^{\ComOp_+}_{\bullet}(A)\in\simp({}_{\ComOp_+}\dg\kk)$
which we deduce from this adjunction relation $\kk/\SSym(-): \kk/\dg\kk\rightleftarrows{}_{\ComOp_+}\dg\kk :\omega$,
for any unitary commutative dg-algebra $A\in{}_{\ComOp_+}\dg\kk$.

We also set $\II\SSym(K) = \bigoplus_{r=1}^{\infty}(K^{\otimes r})_{\Sigma_r}$, for any dg-module $K\in\dg\kk$.
(This object $\II\SSym(K)$ represents the augmentation ideal of the symmetric algebra $\SSym(K)$.)
We have, as in the case of algebras over the Barratt-Eccles operad, an obvious isomorphism of categories
between the category of non-unitary commutative dg-algebras and the category of augmented unitary commutative dg-algebras.
To be explicit, to any non-unitary commutative dg-algebra $A\in{}_{\ComOp}\dg\kk$,
we associate the unitary commutative dg-algebra
such that $A_+ = \kk\oplus A$.
We have:
\begin{equation*}
\Res^{\ComOp_+}_n(A_+) = \SSym(\underbrace{\II\SSym\circ\dots\circ\II\SSym}_n(A)),
\end{equation*}
for any dimension $n\in\NN$, when we consider a unitary commutative algebra of this form $A_+ = \kk\oplus A$.

We now assume that the ground ring $\kk$ is a field of characteristic zero.
We have in this context a model structure on the category ${}_{\ComOp_+}\dg\kk$
which we define by the standard adjunction process,
by assuming that the forgetful functor $\omega: {}_{\ComOp_+}\dg\kk\rightarrow\dg\kk$
creates weak-equivalence and fibrations.
We have a weak-equivalence $\epsilon: \EOp_+\xrightarrow{\sim}\ComOp_+$, between the unitary version of the Barratt-Eccles operad $\EOp_+$
and the unitary version of the commutative operad $\ComOp_+$, which is given by the obvious extension of the augmentation
of the Barratt-Eccles operad $\epsilon: \EOp\xrightarrow{\sim}\ComOp$.
We consider the extension and restriction functors
associated to this weak-equivalence
of operads $\epsilon_!: {}_{\EOp_+}\dg\kk\rightleftarrows{}_{\ComOp_+}\dg\kk :\epsilon^*$.
We use that these adjoint functors define a Quillen equivalence (see for instance~\cite[Theorem 16.A]{OperadModules}).
We also have $\epsilon_!\ESym_+(-) = \SSym(-)$ and $\epsilon_!(\kk/\ESym_+(-)) = \kk/\SSym(-)$ by adjunction.
We consequently have the relation $\epsilon_!\Res^{\EOp_+}_{\bullet}(A) = \Res^{\ComOp_+}_{\bullet}(A)$
for any unitary commutative dg-algebra $A\in{}_{\ComOp_+}\dg\kk$.
We still have $\epsilon_!|\Res^{\EOp_+}_{\bullet}(A)| = |\Res^{\ComOp_+}_{\bullet}(A)|$
when we pass to the geometric realization, because the left adjoint functor of a Quillen adjunction
preserves cosimplicial frames (and coends). Theorem~\ref{UnitarySetting:MainResult}
therefore admits the following corollary:

\begin{thm}\label{CommutativeAlgebras:MainResult}
We assume that the ground ring $\kk$ is a field of characteristic zero. We consider the geometric realization of the cotriple resolution
of a unitary commutative dg-algebra $A\in{}_{\ComOp_+}\dg\kk$.
We then have a weak-equivalence $\epsilon: |\Res^{\ComOp_+}_{\bullet}(A)|\xrightarrow{\sim} A$,
so that $|\Res^{\ComOp_+}_{\bullet}(A)|$ defines a cofibrant resolution of our object $A$
in the model category of unitary commutative dg-algebras ${}_{\ComOp_+}\dg\kk$
as soon as $A\not=0$.
\qed
\end{thm}

In rational homotopy theory, it is more natural to deal with commutative algebras in non-negatively upper graded dg-modules (see~\cite{BousfieldGugenheim,Sullivan}).
In what follows, we rather use the notation $\dg^*\kk$ for the category of non-negatively upper graded dg-modules.
We similarly use the notation ${}_{\ComOp_+}\dg^*\kk$ for the category of commutative algebras in $\dg^*\kk$
and we call unitary commutative cochain dg-algebras the objects
of this category of commutative algebras $A\in{}_{\ComOp_+}\dg^*\kk$.
We still have a model structure on the category of cochain commutative dg-algebras ${}_{\ComOp_+}\dg^*\kk$
when the ground ring is a field of characteristic zero (see~\cite[\S 4]{BousfieldGugenheim} or \cite[\S II.6]{FresseBook}).

We have an obvious embedding of categories $\iota^*: {}_{\ComOp_+}\dg^*\kk\hookrightarrow{}_{\ComOp_+}\dg\kk$
which identifies any object of the category of unitary commutative cochain dg-algebras $A\in{}_{\ComOp_+}\dg^*\kk$
with a unitary commutative dg-algebra $A\in{}_{\ComOp_+}\dg\kk$ such that $A_n = 0$ for $n>0$
and $A_n = A^{-n}$ for $n\leq 0$.
We may check that this functor admits a left adjoint $\tau_{\sharp}^*: {}_{\ComOp_+}\dg\kk\rightarrow{}_{\ComOp_+}\dg^*\kk$
so that we have a Quillen adjunction $\tau_{\sharp}^*: {}_{\ComOp_+}\dg\kk\rightleftarrows{}_{\ComOp_+}\dg^*\kk :\iota$
between our model category of unitary commutative dg-algebras ${}_{\ComOp_+}\dg\kk$
and the model category of unitary commutative cochain dg-algebras ${}_{\ComOp_+}\dg^*\kk$.
We can use this Quillen adjunction to transport our result on the cotriple resolution of unitary commutative dg-algebras
to the category of unitary commutative cochain dg-algebras.
We actually need to restrict ourselves to the case of unitary commutative cochain dg-algebras such that $\DGH^0(A) = \kk$
in order to make this argument work.
We therefore need this extra assumption $\DGH^0(A) = \kk$
to get a counterpart of the statement of Theorem~\ref{CommutativeAlgebras:MainResult}
in the model category of unitary commutative cochain dg-algebras.

\section*{Outlook}

To complete this account, let us mention that our main results extend to the categories of dg-algebras over operads $\POp$
which are cofibrant as symmetric sequences (we also say that $\POp$ is $\Sigma$-cofibrant in this case).
This cofibration assumption just ensures that the category of dg-algebras associated to our operad $\POp$
is equipped with a valid (semi)model structure (see~\cite{OperadModules} for a general reference on this subject).
In fact, we can reduce our verifications to the case of operads $\ROp$ equipped with a coproduct $\Delta: \ROp\rightarrow\ROp\otimes\EOp$
which takes values in the arity-wise tensor product with the Barratt-Eccles operad $(\ROp\otimes\EOp)(r) = \ROp(r)\otimes\EOp(r)$, $r>0$,
and which lifts the morphism $\epsilon: \ROp\otimes\EOp\xrightarrow{\sim}\ROp$
induced by the augmentation morphism of the Barratt-Eccles operad $\epsilon: \EOp(r)\rightarrow\kk$
in each arity $r>0$. Indeed, to check the general result from this particular case,
we can use that an operad $\POp$ is weakly-equivalent to an operad $\ROp$
equipped with such a coproduct (see~\cite[\S 3.1]{BergerFresse})
and the existence of a weak-equivalence $\epsilon: \ROp\xrightarrow{\sim}\POp$ implies that we have a Quillen equivalence
at the level of the model categories of algebras associated to our operads.
In~\S\ref{Background}, we basically use that the Barratt-Eccles operad $\EOp$ is equipped with a coproduct
to define our cosimplicial framing functor on the category of $\EOp$-algebras.
Now, we can use the coproduct operation $\Delta: \ROp\rightarrow\ROp\otimes\EOp$ to extend the construction
of this cosimplicial framing to $\ROp$-algebras in dg-modules
and we readily check that the rest of our argument lines
remains valid in this context.

\bibliographystyle{plain}
\bibliography{CotripleResolution}

\end{document}